\documentclass{article}
\usepackage{amsmath}
\usepackage[left=3.0cm, right=3.0cm, bottom=2.7cm, top=2.7cm]{geometry}
\usepackage{hyperref}
\usepackage{cleveref}
\usepackage{amssymb}
\usepackage{mathtools}
\usepackage{tikz-cd}
\usepackage{amsthm}
\usepackage{graphicx}
\usepackage{graphics}
\usepackage{mathrsfs}
\usepackage{cleveref}
\usepackage{thmtools}
\usepackage{enumitem}
\usepackage{bbm}
\usepackage{dsfont}
\usepackage[toc,page]{appendix}

\newlist{steps}{enumerate}{1}
\usepackage{mathtools}
\makeatletter
\newcommand{\xMapsto}[2][]{\ext@arrow 0599{\Mapstofill@}{#1}{#2}}
\def\Mapstofill@{\arrowfill@{\Mapstochar\Relbar}\Relbar\Rightarrow}
\makeatother
\usepackage{array}
\setlist[steps, 1]{label = Step \arabic*:}
\theoremstyle{plain}
\usepackage{footnote}
\makesavenoteenv{tabular}
\newtheorem*{theorem*}{Theorem}
\newtheorem{theorem}{Theorem}[subsection]
\newtheorem{definition}[theorem]{Definition}

\newtheorem{proposition}[theorem]{Proposition}
\newtheorem{remark}[theorem]{Remark}
\newtheorem{corollary}[theorem]{Corollary}
\newtheorem*{corollary*}{Corollary}
\newtheorem*{proposition*}{Proposition}
\newtheorem{definition*}{Definition}
\newtheorem{exmp}[theorem]{Example}

\newtheorem{lemma}[theorem]{Lemma}
\newtheorem*{lemma*}{Lemma}

\numberwithin{equation}{subsection}
\usepackage{rotating}
\usepackage{comment}
\usepackage{mathabx,epsfig}

\usepackage[nottoc,notlot,notlof]{tocbibind} 

\title{Quantization of Holomorphic Symplectic Manifolds: Analytic Continuation of Path Integrals and Coherent States}
\author{Joshua Lackman\footnote{josh@pku.edu.cn}}
\date{}
\begin{document}
\maketitle
\begin{abstract}
\noindent
We extend Berezin's quantization $q:M\to\mathbb{P}\mathcal{H}$ to holomorphic symplectic manifolds, which involves replacing the state space $\mathbb{P}\mathcal{H}$ with its complexification $\textup{T}^*\mathbb{P}\mathcal{H}.$ We show that this is equivalent to replacing rank–1 Hermitian projections with all rank–1 projections. We furthermore allow the states to be points in the cotangent bundle of a Grassmanian. We also define a holomorphic path integral quantization as a certain idempotent in a convolution algebra and we prove that these two quantizations are equivalent. For each $n>0,$ we construct a faithful functor from the category of finite–dimensional $C^*$–algebras to to the category of hyperk\"{a}hler manifolds and we show that our quantization recovers the original $C^*$–algebra.\ In particular, this functor comes with a homomorphism from the commutator algebra of the $C^*$–algebra to the Poisson algebra of the associated hyperk\"{a}hler manifold.\ Related to this, we show that the cotangent bundles of Grassmanians have commuting almost complex structures that are compatible with a holomorphic symplectic form.
\end{abstract}
\tableofcontents
\section{Introduction}
The main ingredient of Berezin's quantization (\cite{ber1}) of a symplecic manifold $(M,\omega)$ is a pointwise quantization map
\begin{equation}\label{ber}
    q:M\to \mathbb{P}\mathcal{H}\subset B\mathcal{H}\;\footnote{Points in $\mathbb{P}\mathcal{H}$ are identified with Hermitian rank–1 projections. States in the image of $q$ are the coherent states.}
\end{equation}
that has the overcompleteness property, which is (essentially) given by
\begin{equation}
    \mathds{1}_{\mathcal{H}}=\int_M q\,\omega^{\textup{top}}\;.
\end{equation}
This is related to geometric quantization in that $q$ determines a prequantum line bundle with a distinguished subspace of sections. However, unlike geometric quantization which frequently only quantizes constant functions,\footnote{With respect to the K\"{a}hler polarization on a complex manifold with no infinitesimal symmetries, geometric quantization only quantizes constant functions. In complete generality the subspace of quantizable functions is very small.} Berezin's quantization quantizes all integrable functions — it also quantizes states, something that geometric quantization doesn't attempt to do. Such a quantization is rigorously equivalent to a path integral quantization (\cite{Lackman1},\cite{poland0}), which generalizes to Poisson manifolds. The standard example of Berezin's quantization is Toeplitz quantization (\cite{bord}), but it works even in cases when no polarization exists (\cite{bor},\cite{kor}).
\\\\In this paper, we extend Berezin's quantization to the holomorphic category and prove that it is equivalent to a path integral quantization. The basic idea is that while the pure quantum states of a real sympelctic manifold correspond to rank–1 \textit{Hermitian} projections, the pure quantum states of a holomorphic symplectic manifold correspond to \textit{all} rank–1 projections. Furthermore, we extend the codomain of the quantization map \ref{ber} to a more general space of density operators. Together, these extensions are equivalent to replacing $\mathbb{P}\mathcal{H}$ with the cotangent bundle of a Grassmanian. This is because there is a natural embedding
\begin{equation}\label{cotem}
    \textup{T}^*\mathbf{G}_n\mathcal{H}\xhookrightarrow{} B\mathcal{H}
\end{equation}
as rank–n projections. This embedding allows for the construction of the hyperk\"{a}hler structure of $\textup{T}^*\mathbf{G}_n\mathcal{H}$ (\cite{biq}) at the level of $C^*$–algebras.\ However, for the purposes of this paper it isn't \textit{exactly} the hyperk\"{a}hler structure that's important.\ Rather, we focus on a pair of integrable, \textit{commuting} almost complex structures $I,J$ together with an $I$–holomorphic symplectic form $\Omega$ that satisfies $\Omega(JX,JY)=\Omega(X,Y).$ All constructions in this paper are at the level of $C^*$-algebras and are therefore functorial.
\\\\We state some of our results. First, as evidence of the efficacy of using Grassmanians to quantize, we have the following (where we use the embedding \cref{cotem}):
\begin{proposition}
The expectation value map 
\begin{equation}
    \langle\cdot\rangle:B\mathcal{H}\to C^{\omega}(\textup{T}^*\mathbf{G}_n\mathcal{H})\;,\footnote{$C^{\omega}$ denotes analytic functions, $\mathbf{G}_n\mathcal{H}$ is the Grassmanian of $n$–dimensional subspaces, $B\mathcal{H}$ denotes bounded operators.}\;\;\langle M\rangle(q) =\frac{1}{n}\textup{Tr}(qM)
\end{equation}
is a morphism of algebras from the commutator algebra into the Poisson algebra.\ In particular, it induces a product $\star$ on its image such that $f\star g-g\star f=i\{f,g\}.$ 
\end{proposition}
This next proposition is concerened with the complex and symplectic geometry of $\textup{T}^*\mathbf{G}\mathcal{H}:$
\begin{proposition}
Let $\mathcal{H}$ be a finite dimensional complex Hilbert space.\ Then $\textup{T}^*\mathbf{G}_n\mathcal{H}$ has integrable, commuting almost complex structures $I,J$ and an $I$–holomorphic symplectic form $\Omega$ that satisfies $\Omega(JX,JY)=\Omega(X,Y).$ Furthermore, the fibers of the projection map onto $\mathbf{G}_n\mathcal{H}$ are Lagrangian, and the zero section is K\"{a}hler for the real part of $\Omega$ and Lagrangian for its imaginary part.
\end{proposition}
The next two propositions deal with the quantization of $\textup{T}^*\mathbf{G}\mathcal{H}.$ The first of these concerns its (extended) Berezin quantization:
\begin{proposition}
We have
\begin{equation}
    \mathds{1}_{\mathcal{H}}=\int_{q\in\mathbf{G}\mathcal{H}}q\,\Omega^{\textup{top}}(q)\;.
\end{equation}
\end{proposition}
This final proposition is concerned with the corresponding path integral quantization:
\begin{proposition}
The following is a rank–n $I$–holomorphic vector bundle (the complexified tautological bundle) 
\begin{equation}
\mathcal{E}:=\{(q,v)\in \textup{T}^*\mathbf{G}_n\mathcal{H}\times\mathcal{H}: qv=v\}\to \textup{T}^*\mathbf{G}_n\mathcal{H}\;.
\end{equation}
Furthermore, there is an $I$–holomorphic section $P$ of
\begin{equation}
\mathcal{E}^*\boxtimes\mathcal{E}\to \textup{T}^*\mathbf{G}_n\mathcal{H}\times\textup{T}^*\mathbf{G}_n\mathcal{H}\;\footnote{$\mathcal{E}^*\boxtimes\mathcal{E}:=\pi_0^*\mathcal{E}^*\otimes\pi_1^*\mathcal{E},$ where $\pi_0,\pi_1$ are the projection maps $\textup{T}^*\mathbf{G}_n\mathcal{H}\times\textup{T}^*\mathbf{G}_n\mathcal{H}\to \textup{T}^*\mathbf{G}_n\mathcal{H}.$}
\end{equation}
that is polarized with respect to $(-J,J)$ and has the following properties: 
\begin{enumerate}\label{proppp}
    \item $P$ is the identity map on the diagonal.
    \item $P$ determines a holomorphic connection whose curvature has trace equal to $\Omega.$
    \item $P$ is an idempotent of the convolution algebra of $\mathcal{E}\to \textup{T}^*\mathbf{G}_n\mathcal{H}$, where the integration cycle is the zero section.
    \item $P$ determines a Hilbert subspace of sections of $\mathcal{E}$ that is unitarily equivalent to $\mathcal{H}.$
\end{enumerate}
\end{proposition}
As a result of \cref{proppp} (most importantly, property 3), for $x,y\in \mathbf{G}_n\mathcal{H}$ we (formally) have that
\begin{equation}\label{path}
P(x,y)=\int_{\gamma(0)=x}^{\gamma(1)=y}\mathcal{P}(\gamma)\,\mathcal{D}\gamma\;,
\end{equation}
where the integral is over all paths $\gamma$ in $\mathbf{G}_n\mathcal{H}$ beginning and ending at $x,y,$ and $\mathcal{P}(\gamma)$ denotes parallel transport over $\gamma.$ Since $P$ is holomorphic, it analytically continues this path integral to $\textup{T}^*\mathbf{G}_n\mathcal{H}.$ This path integral is a vector bundle generalization of a version of Feynman's phase space path integral for which $\Delta t\to 0$ (or equivalently, the Hamiltonian is set equal to $0$).\ The latter is commonly known as the coherent state path integral, \cite{klauder}, \cite{klauder2}.
\begin{remark}
It is unusual for vector bundles to be used in quantization rather than just line bundles. Their appearance is due to the consideration of higher rank projections, whereas rank–one projections correspond to line bundles — physically, rank–one projections correspond to wave functions/pure states. Higher rank projections correspond to density operators that are maximally mixed given that their image lie in a specific subspace.
\end{remark} 
\section{Complex and Symplectic Geometry of  $\textup{T}^*\mathbf{G}\mathcal{H}$}
First, we will give a simple geometric construction of the commuting almost complex structures and the compatible holomorphic symplectic form on $\textup{T}^*\mathbf{G}_n\mathcal{H}.$ We will then construct the canonical holomorphic prequantum vector bundle and idempotent section. To do this, we will make use of the identification of $\textup{T}^*\mathbf{G}_n(\mathcal{H})$ with projections on $\mathcal{H}.$
\subsection{Preliminaries: The embedding $\textup{T}^*\mathbf{G}\mathcal{H}\xhookrightarrow{}B\mathcal{H}$}
We assume the standard identification $\textup{T}_V\mathbf{G}_n(\mathcal{H})\cong \textup{Hom}(V,V^\perp)$ (see \cref{vperp}), so that $\textup{T}^*_V\mathbf{G}_n(\mathcal{H})\cong \textup{Hom}(V^\perp,V)$ via the trace map, ie.
\begin{lemma}
$\textup{Hom}(V^\perp,V)\times\textup{Hom}(V,V^\perp)\mapsto\mathbb{C}\,,\,(f,g)\mapsto \textup{Tr}(fg)$ is a perfect pairing.
\end{lemma}
\begin{proof}
$\textup{Tr}(ff^*)>0.$
\end{proof}
\begin{proposition}\label{id}
$\textup{T}^*\mathbf{G}_n(\mathcal{H})\cong \{q\in B\mathcal{H}: q^2=q,\,\textup{Tr}(q)=n\}.$
\end{proposition}
\begin{proof}
We will show that 
\begin{equation}
   \textup{Hom}(V^\perp,V)\cong\{\textup{Projections onto }V\}\;. 
\end{equation}
Let $q$ be any projection onto $V.$ We define $f\in \textup{Hom}(V^\perp,V)$ by
\begin{equation}\label{grass}
f=q\vert_{V^\perp}\;.
\end{equation}
Conversely, let $f\in \textup{Hom}(V^\perp,V)$ and let $q_V$ be the orthogonal projection onto $V.$ We assume that $f$ is defined on all of $\mathcal{H}$ by extending it by $0$ on $V.$ We get a projection onto $V$ by defining
\begin{equation}
    q=q_V+f\;.
\end{equation}
This is the inverse of \cref{grass}.
\end{proof}
With the identification given above, the zero section corresponds to the natural inclusion of orthogonal projections into projections, and the natural map $\textup{T}^*\mathbf{G}\mathcal{H}\to \mathbf{G}\mathcal{H}$ is given by sending a projection to its corresponding orthogonal projection. From now on we will assume this identification — its main advantage is that we get a very simple description of $\textup{T}(\textup{T}^*\mathbf{G}\mathcal{H}):$
\begin{lemma}\label{lemcot}
Under the identification of \cref{id},
\begin{equation}\label{tan}
 \textup{T}(\textup{T}^*\mathbf{G}\mathcal{H})\cong \{(q,A)\in \textup{T}^*\mathbf{G}\mathcal{H}\times B\mathcal{H}: qA+Aq=A\}\;.
\end{equation}
\end{lemma}
\begin{proof}
This follows from standard arguments by observing that $A\in B\mathcal{H}$ is tangent to the space of projections at $P$ if and only if  $(q+\varepsilon A)^2=q+\varepsilon A+\mathcal{O}(\varepsilon^2).$
\end{proof}
\begin{remark}\label{vperp}
The identification $\textup{T}_V\mathbf{G}_n(\mathcal{H})\cong \textup{Hom}(V,V^\perp)$ is given as follows: $A\in \textup{T}_q\mathbf{G}\mathcal{H}$ satisfies $Aq+qA=A,$ which implies that $qAq=0,$ ie. $A\vert_{q(\mathcal{H})}\in \textup{Hom}(q(\mathcal{H}),q(\mathcal{H})^\perp).$ The pairing $\textup{T}_V^*\mathbf{G}\mathcal{H}\times \textup{T}_V\mathbf{G}\mathcal{H}\to\mathbb{R}$ is given by $(q,A)\mapsto \textup{Tr}(qA).$
\end{remark}
\subsection{The Complex Structures and Symplectic Form}
\subsubsection{Complex Structures}
We will first describe a pair of commuting almost complex structures and then we will describe the holomorphic symplectic form. The proofs that $I,J$ are integrable will be done in \cref{cg}.
\begin{definition}
Let $\mathcal{H}$ be a complex Hilbert space. We define (integrable, $\textup{GL}(\mathcal{H})$–equivariant) \textbf{commuting} almost complex structures $I, J$\footnote{The algebra of commuting complex structures (rather than anticommuting) is called the tessarines (rather than the quarternions). Instead of $(IJ)^2=-1$ we have $(IJ)^2=1.$} on $\textup{T}^*\mathbf{G}\mathcal{H}$ by 
\begin{equation}
    (q,A)\xmapsto{I} (q,iA)\;,\;\; (q,A)\xmapsto{J} (q,i[A,q])\;.
    \end{equation}
\end{definition}
We will prove that $J$ is an almost complex structure:
\begin{proof} To see that $J^2=-1,$ we observe that $Aq+qA=A$ and $q^2=q$ imply that $qAq=0.$ From this we get
\begin{align}
& \nonumber J^2A=i[i[A,q],q]=-[Aq-qA,q]=-(Aq-qA)q-q(Aq-qA)
\\&=-Aq-qA=-(Aq+qA)=-A\;.
\end{align}
To see that $J$ preserves the tangent spaces, we have
\begin{equation}\label{map}
   (JA)q+q(JA)= i[A,q]q+qi[A,q]=iAq-iqA=JA\;.
\end{equation}
\end{proof}
This previous definition is the only one in this paper that requires that $\mathcal{H}$ is a complex Hilbert space. In fact, a large part of this paper doesn't even explicitly use the inner product. 
\begin{definition}
We define an involution of the tangent bundle of $\textup{T}^*\mathbf{G}\mathcal{H}$ by $K:=IJ.$
\end{definition}
Note that $K$ exists in the case that $\mathcal{H}$ is a real Hilbert space as well.
\begin{remark}\label{jq}
For any $q\in \textup{T}^*\mathbf{G}\mathcal{H},$  $I_q,J_q,K_q:B\mathcal{H}\to \textup{T}_q^*\mathbf{G}\mathcal{H}$ are surjections, with the same formulas given as above.
\end{remark}
\begin{proposition}
The space of \textit{orthogonal} projections is a totally real submanifold with respect to $I,$ ie. for $q\in \mathbf{G}\mathcal{H},$
\begin{equation}
\textup{T}_q\textup{T}^*\mathbf{G}\mathcal{H}=\textup{T}_q\mathbf{G}\mathcal{H}\oplus I(\textup{T}_q\mathbf{G}\mathcal{H})\;.\footnote{In particular, it is the fixed point set of the adjoint map, which is an antiholomorphic involution.}
\end{equation} In addition, it is a complex submanifold with respect to $J.$ 
    \end{proposition}
\begin{proof} 
If $q\in \mathbf{G}\mathcal{H}$ then $A\in \textup{T}_q\textup{T}^*\mathbf{G}\mathcal{H}$ is in $\textup{T}_q\mathbf{G}\mathcal{H}$ if and only if $A=A^*.$ The statement now follows from the fact that $\textup{T}_q\textup{T}^*\mathbf{G}\mathcal{H}$ decomposes as a direct sum of Hermitian and skew–Hermitian tangent vectors, since $I_q$ permutes Hermitian and skew-Hermitian operators while $J_q$ fixes them.
\end{proof}
Of course, the zero section (ie. orthogonal projections) is the fixed point set of the anti–holomorphic involution $^*.$ 
\\\\As a consequence of the fact that $I, J, K$ commute, we have the following:
\begin{corollary}
$J,K $ preserve the $I$–holomorphic tangent bundle of $\textup{T}^*\mathbf{G}\mathcal{H}.$ Similarly, $I,K$ preserve the $J$–holomorphic tangent bundle of $\textup{T}^*\mathbf{G}\mathcal{H},$ and $I,J$ preserve the $\pm 1$ eigenbundle of $K.$
\end{corollary}
\subsubsection{Symplectic Form}
We now define the canonical holomorphic symplectic form. The proof that it is non–degenerate will be done in \cref{cg}.
\\\\\textbf{Convention:} For the rest of this paper, in the context of $\mathbf{G}_n\mathcal{H}$  the trace will be normalized so that rank–n projections have trace equal to $1.$ 
\begin{definition}\label{sym}
We have an $I$–holomorphic symplectic form $\Omega$ on $\textup{T}^*\mathbf{G}\mathcal{H}$ given by
\begin{equation}
    \Omega_q(A,B)\xmapsto{\Omega}i\textup{Tr}(q[A,B])\;.\footnote{The trace of a simple finite dimensional $C^*$–algebras is (up to a constant) the unique linear functional for which $f(ab)=f(ba).$ Our normalization fixes it. Finite dimensional $C^*$–algebra are direct sums of simple ones.}
\end{equation}
\end{definition}
We discuss its compatibility with $J,K.$
\begin{lemma}\label{inv}
$\Omega(JX,JY)=\Omega(X,Y),\,\Omega(KX,KY)=-\Omega(X,Y),$ and $\Omega$ is invariant under the action of $\textup{GL}(\mathcal{H}),$ ie. $q\mapsto MqM^{-1}.$
\end{lemma}
\begin{proof}
The proof is a direct computation.
\end{proof}
\begin{proposition}
The foliations given by the $\pm 1$ eigenbundles of $K$ are Lagrangian polarizations.
\end{proposition}
\begin{proof}
This follows from the previous lemma, since if $X,Y$ are in these eigenbundles then $\Omega(X,Y)=-\Omega(X,Y).$
\end{proof}
\begin{lemma}
The foliation determined by the $+1$ eigenbundle of $K$ is the foliation given by the projection $\textup{T}^*\mathbf{G}\mathcal{H}\to\mathbf{G}\mathcal{H}.$
\end{lemma}
\begin{proof}
Two points $q,q'$ are in the same fiber of $\textup{T}_n^*\mathbf{G}\mathcal{H}\to\mathbf{G}_n\mathcal{H}$ if and only if they have the same image as projection operators, which is the case if and only if $q'q=q.$ As a result, a vector $(q,A)$ is tangent to $\textup{T}^*\mathbf{G}\mathcal{H}\to\mathbf{G}\mathcal{H}$ if and only if 
\begin{equation}
    (q+\epsilon A)q=q+\mathcal{O}(\epsilon^2)\;,
\end{equation}
which is the case if and only if $Aq=0.$ Hence,
\begin{equation}
    K_q(A):=[q,A]=qA=qA+Aq=A,
\end{equation}
and therefore $(q,A)$ is in the $+1$ eigenbundle of $K.$ On the other hand, if $K_q(A)=A$ then it follows from $qA+Aq=A$ that $Aq=0.$
This completes the proof.
\end{proof}
\begin{corollary}
$\textup{T}^*\mathbf{G}\mathcal{H}\to\mathbf{G}\mathcal{H}$  determines a Lagrangian polarization.
\end{corollary}
\begin{remark}
The two Lagrangian polarizations correspond to the two maps $\textup{T}\textup{T}^*\mathbf{G}\mathcal{H}\to \textup{T}^*\mathbf{G}\mathcal{H}$
given by
\begin{equation}
(q,A)\mapsto q+Aq\,,\,(q,A)\mapsto q+qA\,.
\end{equation}
The adjoint map is a vector bundle map between these.
\end{remark}
\begin{proposition}
The zero section is K\"{a}hler for the real part of $\Omega$ and Lagrangian for its imaginary part. 
\end{proposition}
\begin{proof}
That the zero section is symplectic for the real part and Lagrangian for the imaginary part follows from observing that if $A,B,P$ are Hermitian then $\Omega_q(A,B)$ is real. 
\\\\To further see that it is K\"{a}hler, using the cyclic property of the trace we have that
\begin{equation}
    \Omega_q(A,B)=\textup{Tr}(iA[B,q])=\textup{Tr}(AJB)\;.
\end{equation}
Since $(q,A,B)\mapsto \textup{Tr}(AB)$ is a Riemannian metric on $\mathbf{G}\mathcal{H},$ this completes the proof.
\end{proof}
Of course, in the case that $n=1$ the pullback of $\Omega$ is just the Fubini–Study form.
\subsection{Classical — Quantum Correspondence of Grassmanians}\label{cg}
We will prove that $\Omega$ is non–degenerate and that $I,J$ are integrable. We will also discuss the classical–quantum correspondence of $\textup{T}^*\mathbf{G}\mathcal{H},$ much of which is shown for $\mathbb{P}\mathcal{H}$ in \cite{poland0}. That $\Omega$ is closed follows from \cref{propp}. First we need the following definition, which should be thought of as an expectation value map:
\begin{definition}
We define a unit–preserving map 
\begin{equation}
B\mathcal{H}\xrightarrow[]{\langle\cdot\rangle} C^{\omega}(\textup{T}^*\mathbf{G}\mathcal{H})\;,\;\;\langle M\rangle(q)=\textup{Tr}(qM)\;.\footnote{Recall that we have normalized the trace so that, in the context of $\mathbf{G}_n\mathcal{H},$ rank–n projections have trace $1.$}
\end{equation}
\end{definition}
\begin{lemma}\label{hamil}
The Hamiltonian vector of $\langle M\rangle$ is given by 
\begin{equation}
 q\mapsto J_q(M)=i[M,q]  
\end{equation}\;.
\end{lemma}
\begin{proof}
This follows from the fact that for $B\in\textup{T}_q\textup{T}^*\mathbf{G}\mathcal{H},$
\begin{equation}
\Omega_q(i[M,q],B)=\textup{Tr}(M(qB+Bq))=\textup{Tr}(MB)=d\langle M\rangle (B)\;.
\end{equation}
\end{proof}
Of course, we still haven't proved that $\Omega$ is symplectic, but the concept of Hamiltonian vector fields makes sense regardless.
\begin{corollary}
$\Omega$ is non–degenerate.
\end{corollary}
\begin{proof}
This follows from the previous lemma and the fact that, for $0\ne A\in\textup{T}_q\textup{T}^*\mathbf{G}\mathcal{H},$
\begin{equation}
    d\langle A^*\rangle(A)= \textup{Tr}(AA^*)>0\;.
    \end{equation}
\end{proof}
It is well–known that it is impossible to have a physically reasonable morphism from an algebra of Poisson brackets into $B\mathcal{H}.$ However, in the case of $\textup{T}^*\mathbf{G}\mathcal{H}$ a natural morphism does exist in the other direction:
 \begin{lemma}\label{mor}
$\langle\cdot\rangle$ is a morphism of algebras with respect to $i\{\cdot,\cdot\}$, ie.
\begin{equation}
    \langle [M,N]\rangle=i\{\langle M\rangle,\langle N\rangle\}\;.
\end{equation}
Furthermore, 
\begin{equation}
B\mathcal{H}\to \mathcal{X}(\textup{T}^*\mathbf{G}\mathcal{H})\;,\;\;M\to i[q,M]
\end{equation}
is a Lie algebra morphism.
 \end{lemma}
\begin{proof}
The first part follows from \cref{hamil}. The second part is a direct computation.
\end{proof}
As a result, we have an exact equivalence between the quantum physics of $\mathcal{H}$ with respect to quantum states linearly generated by rank–n projections, and the classical physics of $\textup{T}^*\mathbf{G}\mathcal{H}$ with respect to classical observables of the form $\langle M\rangle.$ This is related to the geometrical formulation of quantum mechanics (\cite{kibble}).
\begin{corollary}
$I,J$ are integrable.
\end{corollary}
\begin{proof}
This follows by applying the Newlander–Nirenberg theorem with vector fields of the form $q\mapsto i[M,q],$ using the fact that $[q,AB]=qAq=0$ for $A,B\in \textup{T}_q\textup{T}^*\mathbf{G}\mathcal{H}.$ 
\end{proof}
\begin{definition}\label{nonc}
Equip $B\mathcal{H}$ with the Hilbert–Schmidt inner product.
We define a noncommutative product $\star$ on $\langle B\mathcal{H}\rangle\subset C^{\omega}(\textup{T}^*\mathbf{G}\mathcal{H})$ 
by
\begin{equation}
f\star g = \langle FG\rangle\;,
\end{equation}
where $F,G\in (\textup{Ker}\,\langle\cdot\rangle)^{\perp}$ map to $f,g,$ respectively.
\end{definition}
As a consequence of \cref{mor}:
\begin{corollary}
$f\star g-g\star f=i\{f,g\}.$
\end{corollary}
In the case of $\textup{T}^*\mathbb{P}^1,$ see example 6.18 of \cite{Lackman1}. While $\langle B\mathcal{H}\rangle$ is small, we expect that the standard methods can be employed get a convergent star product on $\mathbb{C}[\langle B\mathcal{H}\rangle],$ see \cite{cahen}. Also see{\cite{mar}, page 25 for a discussion and references.
\subsection{Hyperk\"{a}hler Structure of $\textup{T}^*\mathbf{G}\mathcal{H}$}
We've described two commuting almost complex structures on $\textup{T}^*\mathbf{G}\mathcal{H},$ but it also has the structure of a hyperk\"{a}hler manifold (\cite{biq}). In particular, it has anticommuting almost complex structures. One of these is the $I$ we previously described, and the other is the natural one arising from it being the cotangent bundle of a complex manifold — this other almost complex structure agrees with $J$ on the zero section but it isn't $I$–holomorphic. The $I$–holomorphic symplectic form is still $\Omega.$ This completely specifies the hyperk\"{a}hler structure, but on $\textup{T}^*\mathbb{P}^1$ we have the following explicit description:
\begin{definition}
There is an integrable almost complex structure on $\textup{T}^*\mathbb{P}^1$ that \textbf{anticommutes} with $I,$ given by
\begin{equation}\label{anti}
(q,A)\xmapsto{} \frac{i[q,A^*]}{\sqrt{2\textup{Tr}(q^*q)-1}}\;.
\end{equation}
The associated hyperk\"{a}hler metric is the real part of
\begin{equation}
(q,A,B)\mapsto \frac{\textup{Tr}(A^*B)}{\sqrt{2\textup{Tr}(q^*q)-1}}
\end{equation}
and $\omega_I$ is given by its imaginary part.\ The two other symplectic forms are defined by $\Omega=\omega_{J'}+i\omega_{K'},$ where $\Omega$ is as in the previous section. 
\end{definition}
That \cref{anti} is an almost complex structure isn't obvious, but it follows from direct computation after explicitly determining all $(q,A)\in \textup{T}\mathbb{P}^1.$ It would be interesting to find a similar formula for general $\textup{T}^*\mathbf{G}\mathcal{H}.$
\section{Quantization of $\textup{T}^*\mathbf{G}\mathcal{H}$}
We now define the holomorphic vector bundle over $\textup{T}^*\mathbf{G}\mathcal{H}$ that complexifies the tautological bundle over $\mathbf{G}\mathcal{H}.$ The fiber over $q$ will be naturally identified with the subspace it projects onto. We will also define an idempotent section of its convolution algebra.
\begin{definition}\label{taut}
We define an $I$-holomorphic vector bundle with Hermitian connection 
\begin{equation}
    (\mathcal{E},\nabla,\langle\cdot,\cdot\rangle)\xrightarrow[]{\pi} \textup{T}^*\mathbf{G}\mathcal{H}
    \end{equation}
given by
\begin{equation}
    \mathcal{E}=\{(q,v)\in \textup{T}^*\mathbf{G}\mathcal{H}\times\mathcal{H}: qv=v\}\;.
\end{equation}
The Hermitian metric is given by the inner product of $\mathcal{H}.$ The splittings of 
\begin{align}
& \textup{T}_{(q,v)}\mathcal{E}\to \textup{T}_{q}\textup{T}^*\mathbf{G}\mathcal{H}\;,
\\& \textup{T}\mathcal{H}\xhookrightarrow{}\textup{T}_{(q,v)}\mathcal{E}
\end{align}
are given by 
\begin{align}
    &\label{split} A\mapsto (A,Av)\;,
    \\& (A,w)\mapsto w-Av\;,
\end{align}
where we are using the embedding $\textup{T}\mathcal{E}\xhookrightarrow{}\textup{T}\textup{T}^*\mathbf{G}\mathcal{H}\times \textup{T}\mathcal{H}.$
\end{definition}
\begin{lemma}
\begin{equation}
\textup{T}_{(q,v)}\mathcal{E}=\{(A,w)\in \textup{T}_{q}\textup{T}^*\mathbf{G}\mathcal{H}\times\mathcal{H}:qw+Av=w\}\;.
\end{equation}
\end{lemma}
We now define the idempotent section:
\begin{definition}\label{section}
We have an $I$–holomorphic section $P$ of $\mathcal{E}^*\boxtimes\mathcal{E}\to \textup{T}^*\mathbf{G}(\mathcal{H})\times \textup{T}^*\mathbf{G}\mathcal{H}$\footnote{$\mathcal{E}^*\boxtimes\mathcal{E}:=\pi_0^*\mathcal{E}^*\otimes\pi_1^*\mathcal{E},$ where $\pi_0,\pi_1$ are the projection maps $\textup{T}^*\mathbf{G}\mathcal{H}\times\textup{T}^*\mathbf{G}\mathcal{H}\to \textup{T}^*\mathbf{G}\mathcal{H}.$} defined by
\begin{equation}
   v_{q_1}P(q_1,q_2)=q_2v_{q_1}\;, 
\end{equation}
for any $v_{q_1}\in q_1(\mathcal{H}).$ 
\end{definition}
The next two results are simple computations.
\begin{proposition}\label{psplit}
The splitting induced by $P$ is the same as the splitting induced by $\nabla.$ That is, differentiating $P$ in the first factor at the diagonal $\textup{T}^*\mathbf{G}\mathcal{H}\xhookrightarrow{}\textup{T}^*\mathbf{G}\mathcal{H}\times \textup{T}^*\mathbf{G}\mathcal{H}$ determines the same map as \cref{split}.
\end{proposition}
\begin{lemma}
Using the induced connection, $P$ is polarized along $(-J,J), (-K,K).$\footnote{That is, $P$ is polarized with respect to the totally complex distribution induced by $(-J,J)$ and the real distribution given by the $+1$-eigenbundle of $(-K,K).$}
\end{lemma}
Fixing any point $v_q\in \mathcal{E}$ gives us a section of $\mathcal{E},$ given by $v_qP(q,\cdot).$ Due to the previous lemma:
\begin{corollary}\label{polall}
For any $v_{q_0}\in\mathcal{E},$ the section $q\mapsto v_{q_0}P(q_0,q)$ is polarized with respect to $I, J, K.$
\end{corollary}
\begin{remark}
    In fact, being polarized with respect to any two of $I,J,K$ implies being polarized with respect to the third, so this gives a basic invariance of polarization result. 
    \end{remark}
We now define the 3-point function:
\begin{definition}
We have an $I$–holomorphic function
\begin{equation}
    \Delta:\textup{T}^*\mathbf{G}(\mathcal{H})^3\to\mathbb{C}\;,\;\;\Delta(q_1,q_2,q_3)=\textup{Tr}(q_3q_2q_1)\;.
\end{equation}
\end{definition}
For the following, we note that $P(q_1,q_2)P(q_2,q_3)P(q_3,q_1)=q_1q_3q_2$ is an endomorphism of $\mathcal{E}_{q_1}.$
\begin{lemma}\label{propp}
$P$ has the following properties:
\begin{enumerate}
    \item $P$ equals the identity on the diagonal.
    \item $\textup{Tr}(q(q_1,q_2)P(q_2,q_3)P(q_3,q_1))=\Delta(q_1,q_2,q_3).$
    \item The trace of the curvature of $\nabla$ is $\Omega.$\footnote{Up until now we hadn't shown that $\Omega$ is closed, but this proves it.}
\end{enumerate}
\end{lemma}
\begin{proof}
The first two properties are immediate from the definitions. \Cref{split} and property 2 show that the trace of the curvature can be computed by applying the van Est map (see appendix of \cite{Lackman1}) to $\log{\Delta}.$ Since 
\begin{equation}
    \Delta(q,q,q+\varepsilon A)=n+\mathcal{O}(\varepsilon^2)\;,
    \end{equation}
we find that $\textup{VE}\log{\Delta}(A,B)\vert_q$ is the term proportional to $\varepsilon\varepsilon'$ in 
\begin{align}
&i\big(\textup{Tr}((q+\varepsilon A)(q+\varepsilon' B)q)-\textup{Tr}((q+\varepsilon' B)(q+\varepsilon A)q)
\\& =\varepsilon\varepsilon' i\textup{Tr}(q[A,B])+\textup{higher order terms.}
\end{align}
This proves the claim.
\end{proof}
Finally, we need to prove that $P$ is an idempotent of the convolution algebra of $\mathcal{E}\to \textup{T}^*\mathbf{G}\mathcal{H}.$ This is easiest done using Schur's lemma.
\begin{lemma}\label{iden}
Let $\mathcal{H}$ be a finite dimensional Hilbert space. Then
\begin{equation}\label{constt}
\int_{\mathbf{G}_n\mathcal{H}}q\,\Omega^{\textup{top}}
\end{equation}
is a positive multiple of the identity operator on $\mathcal{H},$ where $\Omega$ is the canonical symplectic form and $q$ identifies a point in $\mathbf{G}_n\mathcal{H}$ with its corresponding orthogonal projection.
\end{lemma}
The proof is a straightforward adaptation of the standard proof of the version of this statement that is specialized to projective space (\cite{klauder4}):
\begin{proof}
In the following, we use that $\Omega$ is invariant under $U(\mathcal{H})$ (\cref{inv}). Let $M\in U(\mathcal{H}).$ We have
\begin{align}
M\int_{\mathbf{G}_n\mathcal{H}}q\,\Omega^{\textup{top}}=M\int_{\mathbf{G}_n\mathcal{H}}M^{*}qM\,\Omega^{\textup{top}}=\int_{\mathbf{G}_n\mathcal{H}}q\,\Omega^{\textup{top}}M\;.
\end{align}
Therefore,
\begin{equation}
\int_{\mathbf{G}_n\mathcal{H}}q\,\Omega^{\textup{top}}
\end{equation}
commutes with the action of an irreducible representation of $\mathcal{H},$ and by Schur's lemma it must be a constant multiple of the identity. Taking traces, we find that 
\begin{equation}
\int_{\mathbf{G}_n\mathcal{H}}q\,\Omega^{\textup{top}}=\frac{n}{\textup{dim}\,\mathcal{H}}\textup{Vol}_{\Omega}(\mathbf{G}_n\mathcal{H})\mathds{1}_{\mathcal{H}}>0\;.
\end{equation}
\end{proof} 
\begin{corollary}\label{corid}
$P$ is an idempotent of the convolution algebra of $\mathcal{E}\to \textup{T}^*\mathbf{G}\mathcal{H},$ with respect to the measure 
\begin{equation}\label{meas}
   \frac{\textup{dim}\,\mathcal{H}}{n\textup{Vol}_{\Omega}(\mathbf{G}_n\mathcal{H})}\Omega^{\textup{top}}
\end{equation}
and the cycle given by the zero section.
\end{corollary}
\begin{proof}
 We need to show that 
 \begin{equation}
     \frac{\textup{dim}\,\mathcal{H}}{n\textup{Vol}_{\Omega}(\mathbf{G}_n\mathcal{H})}\int_{\mathbf{G}_n\mathcal{H}}P(q_1,q)P(q,q_2)\,\Omega^{\textup{top}}(q)=P(q_1,q_2)\;.
 \end{equation}
 Since $P(q_1,q)P(q,q_2)=q_2q,$ this follows immediately from \cref{iden}.
\end{proof}
As a result:
\begin{corollary}\label{formal}\footnote{This is formal.}
    $P$ analytically continues the path integral
\begin{equation}
P(x,y)=\int_{\gamma(0)=x}^{\gamma(1)=y}\mathcal{P}(\gamma)\,\mathcal{D}\gamma\;,
\end{equation}
where the integral is over all paths $\gamma$ in $\mathbf{G}_n\mathcal{H}$ beginning and ending at $x,y,$ and $\mathcal{P}(\gamma)$ denotes parallel transport over $\gamma.$
\end{corollary}
\begin{proof}
Let $x,y\in \mathbf{G}\mathcal{H}$ and let $d\mu$ denote \ref{meas}. Using the fact that $P$ is an idempotent and iterating, we have
\begin{equation}\label{prop}
P(x,y)=\int_{\mathbf{G}\mathcal{H}^{n-1}} \prod_{k=0}^{n-1} \mathcal{P}(x_k,x_{k+1})\,d\mu(x_1)\,\cdots\,d\mu(x_{n-1})\;.
  \end{equation}
On the other hand, let $\gamma:[0,1]\to M$ be a $C^1$-path and let $0=t_0<\cdots<t_n=1.$ Due to \cref{psplit}, we have
\begin{equation}\label{conv}
    \prod_{k=0}^{n-1} P(\gamma_{t_k},\gamma_{t_{k+1}})\xrightarrow[]{\Delta t_i\to 0} \mathcal{P}(\gamma)\;.
\end{equation}
Therefore, taking $n\to\infty$ in (\ref{prop}) gives
\begin{equation}
P(x,y)=\int_{\gamma_0=x}^{\gamma_1=y}\mathcal{P}(\gamma)\,\mathcal{D}\gamma\;.
\end{equation} 
\end{proof} 
\subsubsection{The Hilbert Space and Relation to Geometric Quantization}
We will define an inner product on $I$–holomorphic sections and show that the Hilbert space determined by $P$ is unitarily equivalent to $\mathcal{H}.$ At the end of this subsection we will discuss the compatibility of the representation of $B\mathcal{H}$ on $\mathcal{H}$ with the Kostant–Souriau quantization map.
\begin{definition}
Let $\psi_1,\psi_2$ be $I$–holomorphic sections. We have an inner product given by
\begin{equation}
    \langle \psi_1,\psi_2\rangle=\int_{\mathbf{G}\mathcal{H}}\langle \Psi_1,\Psi_2\rangle_q\,\Omega^{\textup{top}}\;.
\end{equation}
$P$ determines a projection operator $\hat{P},$ ie.
\begin{equation}
    (\hat{P}\psi)(q)=\int_{\mathbf{G}\mathcal{H}}\psi(q')P(q',q)\,\Omega^{\textup{top}}(q')\;,
\end{equation}
and we define the quantum Hilbert space $\mathcal{H}_P$ to be the image of this operator.
\end{definition}
This inner product is non–degenerate because the zero section is a totally real submanifold, and therefore the restriction map of $I$–homorphic sections is injective. This is also the reason that if $\psi\vert_{\mathbf{G}\mathcal{H}}\equiv\hat{P}\psi\vert_{\mathbf{G}\mathcal{H}},$ then $\psi\equiv \hat{P}\psi.$   
\begin{proposition}
$\mathcal{H}_P$ exactly consists of sections of $\mathcal{E}\to \textup{T}^*\mathbf{G}\mathcal{H}$ which are simultaneously polarized with respect to $I,J,K.$
\end{proposition}
\begin{proof}
This follows from the definition and \cref{polall}.
\end{proof}
As mentioned earlier, an $I$-polarized section is $J$–polarized if and only if it is $K$–polarized, so this gives an example where two polarizations produce the same result. 
\begin{remark}
We emphasize the idempotent $P$ over the polarizations since this can exist in contexts for which polarizations don't. We view the choice of polarization as a method of computing such a $P.$
\end{remark}
The following shows that quantizing $\textup{T}^*\mathbf{G}\mathcal{H}$ using the idempotent produces the same result as simply quantizing it to $\mathcal{H}:$ 
\begin{lemma}
The map 
\begin{equation}\label{uni}
\mathcal{H}\to \mathcal{H}_P\;,\;\;\Psi\mapsto \psi_{\Psi}\,,\,\psi_{\Psi}(q)=q\Psi 
\end{equation}
is a unitary equivalence.
\end{lemma}
\begin{proof}
This map has an inverse, given by
\begin{equation}
 \mathcal{H}_P\to \mathcal{H}\;,\;\; \psi\mapsto \int_{\mathbf{G}\mathcal{H}}\psi\,\Omega^{\textup{top}}\;,
\end{equation}
where on the right side we are using the fact that a section maps into $\mathcal{H}.$ That this is a left–inverse follows from \cref{schur}. That it is a right inverse follows from the definition of $P$ and the fact that
\begin{equation}
\psi(q)=\int_{\mathbf{G}\mathcal{H}}\psi(q')P(q',q)\,\Omega^{\textup{top}}(q')\;.
\end{equation}
To see that this map preserves the inner product, we note that the zero section consists of Hermitian projections, and therefore
\begin{align}
& \langle \Psi_1,\Psi_2\rangle=\int_{\mathbf{G}\mathcal{H}}\langle q\Psi_1,q\Psi_2\rangle\,\Omega^{\textup{top}}(q)
\\& = \int_{\mathbf{G}\mathcal{H}}\langle \Psi_1,q\Psi_2\rangle\,\Omega^{\textup{top}}(q)=\langle \Psi_1,\Psi_2\rangle\;,
\end{align}
where the last equality follows from \cref{schur}.
\end{proof}
In the case that $n=1,$ $\Psi\in\mathcal{H}$ can be identified with the physicists' abstract wave function $|\Psi\rangle,$ where for $q\in \mathbf{G}\mathcal{H},$ $\psi_{\Psi}(q):=q\Psi$ would be written as 
\begin{equation}
    \psi(q)=|q\rangle\langle q|\Psi\rangle\;.\footnote{Since physicists tend to work locally they would probably write $\psi_{\Psi}(q)=\langle q|\Psi\rangle,$ but strictly speaking this isn't a section.}
\end{equation}
\begin{proposition}
Under the map \cref{uni}, the embedding $\textup{T}^*\mathbf{G}\mathcal{H}\xhookrightarrow{}B\mathcal{H}$ corresponds to the embedding
\begin{equation}
\textup{T}^*\mathbf{G}\mathcal{H}\xhookrightarrow{}B\mathcal{H}_P\;,\;\;q\mapsto \hat{q}\,,\,(\hat{q}\psi)(q')=\psi(q)P(q,q')\;.
\end{equation}
Equivalently, this is the Riesz–dual of the sequilinear form
\begin{equation}
    (\psi_1,\psi_2)\mapsto \langle\psi_1(q),\psi_2(q)\rangle_q\;.
\end{equation}
\end{proposition}
Finally, we discuss more of the relationship to geometric quantization. Recall that the Hamiltonian vector field of $\langle M\rangle$ is $X_{\langle M\rangle}(q) =-i[q,M]$ (\cref{hamil}).
\begin{proposition}(\cite{poland0})
Let $n=1.$ With respect to the map \cref{uni}, the induced map $B\mathcal{H}\to B\mathcal{H}_P$ sends $M$ to the Kostant–Souriau operator of $\langle M\rangle,$ ie.
\begin{equation}\label{ks}
\psi_{M\Psi}=i\nabla_{X_{\langle M\rangle}}\psi_{\Psi}+\langle M\rangle \psi_{\Psi}\;.
\end{equation}
\end{proposition}
\begin{proof}
We will compute the right side of \ref{ks}. Let $\Psi\in \mathcal{H}.$ Then
\begin{align}
    & -i\nabla_{i[q,M]}\psi_{\Psi}=[q,M]\Psi-[q,M]q\Psi=qM\Psi-qMq\Psi\;.
\end{align}
Since $n=1,$ we have that $qMq=\textup{Tr}(qM)q=\langle M\rangle q,$ and therefore
\begin{align}
\psi_{M\Psi}(q)=-i\nabla_{i[q,M]}\psi_{\Psi}+\langle M\rangle\psi_{\Psi}(q)\;.
\end{align}
\end{proof}
This also shows that Kostant–Souriau's heuristic completely fails for $n>1.$
\subsection{Symplectic Manifolds Comptatible With Commuting Almost Complex Structures}
The symplectic manifolds $(M,\Omega)$ we are studying are analogous to hyperk\"{a}hler manifolds, but for a different kind of hypercomplex structure: there are endomorphisms $I,J,K$ of $TM$ such that 
\begin{equation}\label{tess}
    I^2=J^2=-K^2=-IJK=-1\;.
\end{equation}
That is, $I,J$ are commuting almost complex structures and $K=IJ$ is an involution. Tessarines are hypercomplex numbers $i,j,k$ that satisfy the relations \ref{tess}.
\begin{definition}\label{com}
Let $(M,\Omega)$ be a manifold with a complex symplectic form. We say that commuting, integrable almost complex structures $(I,J)$ are compatible with $\Omega$ if $\Omega$ is holomorphic symplectic with respect to $I$ and if $\Omega(JX,JY)=\Omega(X,Y).$
\end{definition}
\begin{lemma}
Assuming the previous definition, we have that 
\begin{equation}
    \Omega(KX,KY)=-\Omega(X,Y)\;.
    \end{equation}
\end{lemma}
\begin{corollary}
The distributions defined by the $\pm i, \pm 1$ eigenbundles of $J,K,$ respectively, are Lagrangian.
\end{corollary}
More specifically, in the examples we've looked at there exists a submanifold which is Lagrangian with respect to the imaginary part of $\Omega$ and $J$–K\"{a}hler for its real part. 
\begin{exmp}
A trivial but significant example of a triple $(I,J,K)$ satisfying \cref{tess} is given by any complex manifold $(M,I),$ where we let $J:=\pm I$ and $K=\mp 1.$ 
\end{exmp}
\begin{exmp}\label{r4}
An example satisfying \cref{com} different from $\textup{T}^*\mathbf{G}\mathcal{H}$ is given by $\mathbb{R}^4\cong \mathbb{C}^2$ with
\begin{equation}
    \Omega=d(x_1+ix_2)\wedge d(y_1+iy_2)\;.
\end{equation}
Here, 
\begin{align}
    &\nonumber I(\partial_{x_1})=\partial_{x_2}\,,\, I(\partial_{y_1})=\partial_{y_2}\;,
    \\&\nonumber J(\partial_{x_1})=\partial_{y_1}\,,\, J(\partial_{x_2})=\partial_{y_2}\;,
    \\& K(\partial_{x_1})=\partial_{y_2}\,,\,K(\partial_{x_2})=-\partial_{y_1}\;.
    \end{align}
The submanifold $x_2=y_2=0$ is $J$–K\"{a}hler for the real part of $\Omega$ and Lagrangian for its imaginary part. On the other hand, $x_2=y_1=0$ is Lagrangian for its real part and symplectic for its imaginary part, and it comes with a Lagrangian polarization induced by $K.$ 
\\\\We have a trivial $I$–holomorphic line bundle $\mathcal{L}\to\mathbb{C}^2$ with an $I$–holomorphic idempotent section of $\mathcal{L}^*\boxtimes\mathcal{L}\to \mathbb{C}^2\times\mathbb{C}^2,$ given by
\begin{equation}
    P(z,w)=e^{-\frac{1}{4\hbar}(z\bar{z}+w\bar{w}-2i\bar{z}w)}\;,
\end{equation}
where $z=(x_1+ix_2)+i(y_1+iy_2),\,\bar{z}=(x_1+ix_2)-i(y_1+iy_2),$ and similarly for $w.$
\\\\As in the case of $\textup{T}^*\mathbf{G}\mathcal{H},$ there is also an almost complex structure $J'$ that anticommutes with $I,$ it is given by
\begin{equation}
J'(\partial_{x_1})=\partial_{y_1}\,,\,J'(\partial_{x_2})=-\partial_{y_2}\,.
\end{equation}
\end{exmp}
\section{Holomorphic Pointwise Quantization}
We give a definition of quantization of holomorphic manifolds which applies to holomorphic symplectic manifolds and extends that of Berezin quantization (\cite{ber1}) in two ways: one is that it extends the definition from real symplectic manifolds to holomorphic symplectic manifolds, and the other is that it generalizes the definition to allow the codomain to be Grassmanians (over any field) rather than just projective spaces.
\begin{definition}
Let $(Y,\Omega^{\textup{top}})$ be a complex manifold with holomorphic volume form\footnote{A holomorphic volume form is a non–vanishing holomorphic top form, and thus its degree is half the real dimension of the manifold.} and let $M\xhookrightarrow{\iota} Y$ be a submanifold for which $\iota^*\Omega^{\textup{top}}$ is a (possibly complex) volume form. We define a holomorphic quantization of $(Y,\Omega^{\textup{top}})$ to be a holomorphic map $q:Y\xhookrightarrow{} \textup{T}^*\mathbf{G}_n\mathcal{H}\subset B\mathcal{H}$\footnote{Here, $\mathcal{H}$ is a separable Hilbert space.} for which (up to rescaling $\Omega^{\textup{top}}$ by a constant)
\begin{equation}\label{def}
    \mathds{1}_{\mathcal{H}}=\int_M q\,\Omega^{\textup{top}}\;.\footnote{In practice, $q$ tends to be an embedding. We may call such submanifolds overcomplete.}
\end{equation}
In the case of a real manifold with volume form $(M,\Omega^{top}),$ we consider the same definition but only with a map $q:M\to \textup{T}^*\mathbf{G}_n\mathcal{H}.$ 
\\\\We say that the quantization is Hermitian if $q\vert_M$ maps into the zero section.  Furthermore, we call $n$ the rank of the quantization, and we say that say the quantization is pure if $n=1$ and mixed otherwise. 
\end{definition}
In this definition, the equality is interpreted weakly, ie. 
\begin{equation}
    \langle\Psi_1|\Psi_2\rangle=\int_M \langle\Psi_1|q\Psi_2\rangle\,\Omega^{\textup{top}}\;.
\end{equation}
With this definition, Hermitian quanizations map functions that are real–valued on $M$ to Hermitian operators. In this case, $q$ is an analytic continuation of a Berezin quantization.
\\\\In the context of holomorphic symplectic manifolds $(Y,\Omega),$ the idea is  $q$ is a symplectic embedding and that $\Omega^{\textup{top}}$ is the canonical holomorphic volume form. However, since quantizations should come in $\hbar$–families $q_{\hbar},$ this only needs to be approximately true. ie. 
\begin{equation}
\hbar\,q_{\hbar}^*\Omega_{\textup{T}^*\mathbf{G}\mathcal{H}}-\Omega=\mathcal{O}(\hbar)\;,
\end{equation}
This is discussed more in \cite{Lackman1}. In some cases, there will be a map $q:Y\to \textup{T}^*\mathbf{G}_n\mathcal{H}$ for which \cref{def} holds on more than one submanifold. In this case, we get quantizations of real symplectic manifolds with an isomorphism between their Hilbert spaces. This can be thought of as a version of the BKS pairing.
\begin{remark}
Equivalently, we can replace the left side of \cref{def} with an orthogonal projection — in this case, \cref{def} will hold on a subspace. This can be useful for comparing different quantization maps, eg.\ when performing a Toeplitz–like quantization there are many different compatible almost complex structures, but all of the corresponding quantization maps land in a subspace of the Hilbert space of sections of the prequantum line bundle.
\\\\This may also be useful for defining the quantization of infinite dimensional manifolds as a limit of quantizations of finite dimensional submanifolds, eg. if we want to quantize $\textup{T}^*\mathbf{G}\mathcal{H}$ for infinite dimensional $\mathcal{H},$ we can consider the canonical quantizations of an increasing sequence $\textup{T}^*\mathbf{G}\mathcal{H}_1\subset\textup{T}^*\mathbf{G}\mathcal{H}_2\subset\cdots$ for finite dimensional $\mathcal{H}_1\subset\mathcal{H}_2\subset\cdots$ with $\mathcal{H}=\cup_{i}\mathcal{H}_i.$ 
\end{remark}
The following are the generalizations of Berezin's contravariant and covariant symbols, which in \cite{klauder} are called the anti–ordered symbols and ordered symbols, respectively.
\begin{definition}
We have unit-preserving maps
\begin{equation}
    C^{\omega}(Y)\xrightarrow[]{Q} B\mathcal{H}\footnote{This map is  defined on any function in $L^p(M,\iota^*\Omega^{\textup{top}}),\,1\le p\le \infty.$ If $M$ is compact then it is defined everywhere.} \;,\;\;B\mathcal{H}\xrightarrow[]{\langle\,\rangle} C^{\omega}(Y)
\end{equation}
given by 
\begin{align}\label{quant}
    &Q_f=\int_M fq\,\Omega^{\textup{top}}\;,
    \\&\label{quant2} \langle A\rangle(x)=\textup{Tr}(q(x)A)\;.
\end{align}
\end{definition}
The two maps \ref{quant}, \ref{quant2} are dual in the sense that, when defined,
\begin{equation}\label{dual}
\textup{Tr}(AQ_f)=\int_M \langle A\rangle f\,\Omega^{\textup{top}}\;.
\end{equation}
In particular, letting $A$ and $f$ be the units we get the following: \begin{corollary}
$\textup{dim}\mathcal{H}=n\textup{Vol}\,M.$\footnote{Recall that we have normalized the trace so that in the context of $\mathbf{G}_n{\mathcal{H}},$ rank–n projections have trace equal to $1.$} 
\end{corollary} 
\subsection{The Noncommutative Products}
From \cref{quant2}, such a quantization comes with a distinguished subspace of $C^{\omega}(Y)$ given by functions of the form $\langle A\rangle,$ and these inherit a noncommutative product as in \cref{nonc}. Often, there is a second, related product:
\begin{definition}\label{products}
We have two (often equal) dinstinguished subspaces of $C^{\omega}(Y).$ The first is given by the image of $\langle\cdot\rangle$ and the second is given by 
\begin{equation}
    (\textup{Ker}\,Q)^{\perp}\;,\footnote{In practice, $\langle\cdot\rangle$ tends to be injective for $n=1.$}
\end{equation}
where the inner product is that of $L^2(M,\iota^*\Omega^{\textup{top}}).$ There is a product $\star_{\langle\rangle}$ on $\langle B\mathcal{H}\rangle$ given by
\begin{equation}
    f\star_{\langle\rangle} g=\langle FG\rangle\;,
\end{equation}
where $F,G\in (\textup{Ker}\,\langle\cdot\rangle)^{\perp}$ map to $f,g,$ respectively. Here, $\perp$ is taken with respect to the Hilbert–Schmidt inner product.\footnote{$\langle B_{\textup{HS}}\mathcal{H}\rangle=\langle B\mathcal{H}\rangle$ since $qM$ has finite rank for any $M\in B\mathcal{H}$ and $qqM=qM.$}
\\\\Assuming that the image of $Q$ is closed under products, we get a product $\star_Q$ on $(\textup{Ker}\,Q)^{\perp}$ given by
\begin{equation}
    f\star_Q g= Q^{-1}(Q_fQ_g)\;,
\end{equation}
where $Q^{-1}$ is the inverse of $Q\vert_{(\textup{Ker}\,Q)^{\perp}}:(\textup{Ker}\,Q)^{\perp}\to \textup{Im}(Q).$
\end{definition}
Due to the equivalence with path integral quantization, $\star_{Q}$ is a rigorous definition of the product defined in equation 2.10 of \cite{brane}, which is a special case of the path integral approach to Kontsevich's star product via the Poisson sigma model (\cite{kont},\cite{catt}).
\begin{lemma}
If $q$ is a Hermitian quantization, then $(\textup{Ker}\,Q)^{\perp}\xrightarrow[]{Q\vert_{\perp}}(\textup{Ker}\,\langle\cdot\rangle)^{\perp}\,,\,(\textup{Ker}\,\langle\cdot\rangle)^{\perp}\xrightarrow[]{\langle\cdot\rangle_{\perp}} (\textup{Ker}\,Q)^{\perp}$ are adjoints. 
\end{lemma}
\begin{proof}
This follows from a direct computation:
\begin{align}
\textup{Tr}(T^*Q_f)=\int_M \overline{\textup{Tr}(qT)}f\,\Omega^{\textup{top}}=\langle \langle T\rangle,f\rangle_{L^2}\;.
\end{align}    
\end{proof}
As a result, the first distinguished subspace of \cref{products} is dense in the second:
\begin{corollary}
If $q$ is a Hermitian quantization, then $Q_{\perp},\,\langle\cdot\rangle_{\perp}$ have dense images. In particular, if $\mathcal{H}$ is finite dimensional then they are isomorphisms and therefore the image of $Q$ is closed under products.
\end{corollary}
The following result gives a family of natural examples:
\begin{lemma}\label{schur}
Let $M\subset \textup{T}^*\mathbf{G}\mathcal{H}$ be a  be a finite dimensional symplectic submanifold that is invariant under an irreducible representation of $\mathcal{H}.$ Then
\begin{equation}\label{const}
\int_{M}q\,\Omega^{\textup{top}}
\end{equation}
is a positive multiple of the identity operator, where $\Omega$ is the canonical symplectic form.
\end{lemma}
We end this subsection with the definition of the 3–point function.
\begin{definition}
We define a holomorphic function, called the 3–point function, by
\begin{equation}
\Delta:Y^3\to\mathbb{C}\,,\;\Delta(y_1,y_2,y_3)=\textup{Tr}(q(y_3)q(y_2)q(y_1))\;.
\end{equation}
\end{definition}
\begin{remark}
One can often identify certain elements of a $C^*$–algebra with states, eg. on a symplectic manifold, a positive function which integrates to $1$ defines a state, and on a Hilbert space a positive operator with trace $1$ defines a state. If we ignore the Hermitian condition then \cref{dual} implies that $Q$ and $\langle \cdot\rangle$ preserve states. Therefore, we not only get a classical observable $\longleftrightarrow$ quantum observable correspondence, we get a classical state $\longleftrightarrow$ quantum state correspondence. 
\end{remark}
\subsection{Examples}
We can take the irreducible representation to be that of the unitary group, and we get the following class of examples:
\begin{exmp}
Let $\mathcal{H}$ be finite dimensional. Then up to scaling $\Omega$ by a positive constant, an example of a holomorphic quantization of points is given by $Y=\textup{T}^*\mathbf{G}_n\mathcal{H},$ with $M=\mathbf{G}_n\mathcal{H}$ being the zero section and $q$ being the identity map. 
\end{exmp}
\begin{exmp}
The Berezin–Toeplitz quantization of a K\"{a}hler manifold $(M,\omega,I)$ is a rank–one example in the real category. Here, $\mathcal{H}$ consists of the square–integrable K\"{a}hler polarized sections of the prequantum line bundle, and $q(x)$ is given by the Riesz–dual of the sesquilinear form
\begin{equation}
(\Psi_1,\Psi_2)\mapsto \langle \Psi_1(x),\Psi_2(x)\rangle_x\;,
\end{equation}
where $\langle \cdot,\cdot\rangle_x$ is the inner product of the fiber of the prequantum line bundle over $x.$\footnote{That is, $q(x)$ is defined by the equation $\langle \Psi_1,q(x)\Psi_2\rangle_{L^2(M)}=\langle \Psi_1(x),\Psi_2(x)\rangle_x.$} In this case, up to a positive constant the volume form is given by $x\mapsto B(x,x)\omega^n(x),$ where $B$ is the Bergman kernel (the integral kernel for the orthogonal projection onto polarized sections). In cases where there is a transitive action of a Lie group on $(M,\omega,I)$ that lifts to the prequantum line bundle with Hermitian connection, $x\mapsto B(x,x)$ is constant. This is discussed more in \cite{Lackman1}.
\end{exmp}
\begin{exmp}
Let $G$ be a compact Lie group. Let
\begin{equation}
    \pi:G\to U(\mathcal{H})
\end{equation}
be an irreducible unitary representation. Let $\mathcal{A}\subset \mathcal{H}$ be a closed subspace and let $H\subset G$ be the subgroup of elements that fix $\mathcal{A}.$ Let $M=G/H$ and let $dx$ the induced left invariant measure.\ For $n=\textup{dim}\,\mathcal{A},$ we have an injective map
\begin{equation}
    q:G/H\to \mathbf{G}_{n}\mathcal{H}\;,
\end{equation}   
where $q(g\mathcal{H})$ is the orthogonal projection onto $U(gH)\mathcal{A}.$ 
\begin{proposition}
Up to rescaling $dx$ by a positive constant,
\begin{equation}
    \mathds{1}_{\mathcal{H}}=\int_{G/H}q\,dx\;.
\end{equation}    
\end{proposition}
\begin{proof}
This follows from Schur's lemma, as in \cref{iden}.
\end{proof}
\end{exmp}
\section{Holomorphic Path Integral Quantization}
In the case of real Poisson manifolds, the following definition is discussed in \cite{Lackman1}.
\begin{definition}\label{patho}
Let $\mathcal{E}\to Y$ be a holomorphic vector bundle with Hermitian metric and let $\Omega^{\textup{top}}$ be a holomorphic volume form. Let $M\xhookrightarrow{\iota} Y$ be a submanifold for which $\iota^*\Omega^{\textup{top}}$ is a (possibly complex) volume form. We say that a holomorphic section $P$ of
\begin{equation}
    \mathcal{E}^*\boxtimes \mathcal{E}\to Y\times Y\;\footnote{Recall that $ \mathcal{E}^*\boxtimes \mathcal{E}:=\pi_0^*\mathcal{E}^*\otimes \pi_1^*\mathcal{E}\to Y\times Y.$}
    \end{equation}
is an (equal–time) propagator if
\begin{enumerate}
    \item \label{nor}$P(x,x)=\mathds{1}$ for all $x\in Y\,.$
     \item\label{bounded} For all $x\in Y,\,|P(x,\cdot)|,|P(\cdot,x)|\in L^2(Y,\Omega^{\textup{top}})$\footnote{The integration is done over $M.$} 
    \item \label{conditionf}For all $x,y\in Y,$ $\int_M P(x,z)P(z,y)\,\Omega^{\textup{top}}_z=P(x,y)$ (up to scaling $\Omega$ by a constant),
    \item \label{norm2} If $x\ne y$ then for any $v_x\in\mathcal{E}_x,v_y\in\mathcal{E}_y,$ there exists $z$ such that $v_xP(x,z)\ne v_yP(y,z)\,.$
  \end{enumerate}
If $P(x,y)=P^*(y,x)$ for all $x,y\in M$ then we say that $P$ is Hermitian.
\\\\Furthermore, we say that $P$ integrates $\nabla_{P},$ where $\nabla_P$ is the connection on $\mathcal{E}$ determined by $P.$
\end{definition}
The first condition says that states are normalized; the second condition is a technical one and it implies that the Hilbert space is non–trivial. The third condition says that $P$ computes a path integral; the fourth condition is an injectivity condition. Under very mild assumptions, $P$ is the integral kernel of a bounded operator on $L^2(\mathcal{E}\to Y,\Omega^{\textup{top}}),$ in which case it defines a projection operator — if $P$ is Hermitian then this is an orthogonal projection.
\begin{remark}
Succinctly, $P$ is a normalized idempotent of the convolution algebra of a vector bundle over the pair groupoid. It generalizes to Poisson manifolds as a normalized idempotent of the  convolution algebra of a prequantum vector bundle over the symplectic groupoid.
\end{remark}
\begin{proposition}
If $P$ is Hermitian then for all $x\in Y,$ $y\mapsto |P(x,y)|^2$ is a probability densiy on $M,$ ie.
\begin{equation}
   \int_M \textup{Tr}\big(P(x,y)P^*(x,y)\big)\,\Omega_y^{\textup{top}} =1\;.\footnote{That is, we are equipping $\textup{hom}(\mathcal{E}_x,\mathcal{E}_y)$ with the Hilbert–Schmidt norm. Again, we have normalized the trace so that rank–n projections have trace equal to $1.$}
\end{equation}
\end{proposition}
\begin{proof}
This follows from conditions 1 and 3 by letting $x=y.$
\end{proof}
For the following, note that $P(y_1,y_2)P(y_2,y_3)P(y_3,y_1)$ defines an endomorphism of $\mathcal{E}_{y_1}.$
\begin{definition}
We define a holomorphic function, called the $3$–point function, to be 
\begin{equation}
\Delta:Y^3\to\mathbb{C}\;,\;\Delta(y_1,y_2,y_3)=\textup{Tr}(q(y_1,y_2)P(y_2,y_3)P(y_3,y_1))\;.
\end{equation}
\end{definition}
\subsection{Equivalence of the Two Quantizations}
We take morphisms in the categories of quantizations to be isomorphisms:
\begin{definition}
The morphisms of a Hermitian, homolomorphic quantization of points are unitary equivalences of Hilbert spaces, ie.\ the morphism associated to a unitary equivalence $U$ is 
\begin{equation}
   q\mapsto UqU^*\;. 
\end{equation}
We take a morphism of Hermitian propagators to be a holomorphic vector bundle map that is a pointwise unitary equivalence, ie. the morphism associated to such a vector bundle map $f$ is 
\begin{equation}
    P\mapsto fPf^*,\, fPf^*(x,y)=f(x)P(x,y)f^*(y)\;.
\end{equation}
\end{definition}
For the special case of real symplectic manifolds with a quantization map valued in $\mathbb{P}\mathcal{H},$ the following is discussed in \cite{poland0} and expanded on in \cite{Lackman1}.
\begin{theorem}\label{equivalence}
Let $(Y,\Omega^{\textup{top}})$ be a complex manifold with a holomorphic volume form and let $M\xhookrightarrow{i}Y$ be a compact submanifold such that $\iota^*\Omega^{\textup{top}}$ is a volume form.
There is a $\Delta$–preserving equivalence of categories between:
\begin{enumerate}
    \item Hermitian Quantizations of points.
    \item Holomorphic vector bundles with a Hermitian metric and Hermitian propagator.
\end{enumerate} 
\end{theorem}
Compactness isn't necessary, it only makes the statement and proof a bit simpler. If we drop the Hermitian condition then the result still holds if we consider vector spaces rather than Hilbert spaces.
\begin{proof}
Given a quantization of points $q:Y\to \textup{T}^*\mathbf{G}\mathcal{H},$ we get a holomorphic vector bundle with Hermitian metric and propagator by pulling back the complexified tautological bundle and propagator of $\textup{T}^*\mathbf{G}\mathcal{H},$ \cref{taut}, \cref{section}. That the pullback is an idempotent follows immeditately from the fact that $q$ satisfies the overcompleteness condition. Conversely, given a holomorphic vector bundle with Hermitian connection and propagator $(\mathcal{E}\to Y, \langle\cdot,\cdot\rangle, P),$ we let $\mathcal{H}_P$ be the Hilbert space given by square–integrable sections $\psi$ such that
\begin{equation}
    \psi=\int_M \psi(x)P(x,\cdot)\Omega^{\textup{top}}_x\;.
\end{equation}
We define $q:Y\to \textup{T}^*\textup{G}\mathcal{H}$ by
\begin{equation}
(q(x)\psi)(y)=\psi(x)P(x,y)
\end{equation}
for $\psi\in\mathcal{H}_P.$ That this is a projection follows from condition \ref{nor}. That $q(x)$ has rank $n$ follows from the fact that, if $v_{x,1},\ldots, v_{x,n}$ is a basis for $\mathcal{E}_x,$ then $v_{x,1}P(x_1,\cdot),\ldots,v_{x,n}P(x,\cdot)$ is a basis for the image of $q(x).$ That $q(x)$ is injective follows from condition \ref{norm2}. To see that $q(x)$ is Hermitian for $x\in M,$ we have that
\begin{align}
 &   \nonumber \langle \psi_1,q(x)\psi_2\rangle_{L^2}=\int_M\langle \psi_1(y),\psi_2(x)P(x,y)\rangle_y\,\Omega^{\textup{top}}_y
 \\&=\int_M\langle \psi_1(y)P(y,x),\psi_2(x)\rangle_y\,\Omega^{\textup{top}}_x=\langle \psi_1(x),\psi_2(x)\rangle_x\;.
\end{align}
In particular, this implies that $ \langle q(x)\psi_1,q(x)\psi_2\rangle_{L^2}=\langle\psi_1,q(x)\psi_2\rangle_{L^2},$ so that $q(x)$ is Hermitian. That $q$ satisfies the overcompleteness condition follows immediately from the fact that $P$ satisfies the idempotent condition.
\\\\The unitary equivalence between the Hilbert spaces of the two quantizations is given by
\begin{equation}
   \mathcal{H}\mapsto \mathcal{H}_P\;,\;\; \Psi\mapsto \psi\,,\,\psi(x)= q(x)\Psi\;.
\end{equation}
That these functors preserve $\Delta$ follows immediately from the definitions.
\end{proof}
\begin{remark}
In some examples there is more than one choice of $M$ that is compatible with $P.$ In such cases, the underlying vector spaces of the quantizations (ie. $\mathcal{H}_P$) are isomorphic. 
\end{remark}
\subsection{Examples}
\begin{exmp}
Returning to example \ref{r4}, an example different from $\textup{T}^*\mathbf{G}\mathcal{H}$ is given by $\mathbb{R}^4\cong \mathbb{C}^2$ with 
\begin{equation}
    \Omega=\frac{1}{\hbar}d(x_1+ix_2)\wedge d(y_1+iy_2)\;.
\end{equation}
Here, the holomorphic line bundle with Hermitian metric is the trivial one, and the connection 1–form is given by
\begin{equation}
    A=\frac{i}{2\hbar}\big((x_1+ix_2)\,d(y_1+iy_2)-(y_1+iy_2)\,d(x_1+ix_2)\big)\;.
\end{equation}
The idempotent section is given by
\begin{equation}
    P(z,w)=e^{-\frac{1}{4\hbar}(z\bar{z}+w\bar{w}-2\bar{z}w)}\;,
\end{equation}
where 
\begin{equation}
z=(x_1+ix_2)+i(y_1+iy_2)\;,\;\;\bar{z}=(x_1+ix_2)-i(y_1+iy_2)\;,
\end{equation}
and similarly for $w.$
Linear 2–dimensional submanifolds provide many submanifolds that can be used for $M,$ where for convergence of the integral we need that, generically, \begin{equation}\label{conv}
    \textup{Re}(z\bar{z})> 0\;.
\end{equation}
For example, we can let $M=\{x_2=y_2=0\},$ which is K\"{a}hler for the real part of $\Omega$ and Lagrangian for its imaginary part — this gives a Hermitian quantization. Sections that satisfy
\begin{equation}
    \psi=\int_{\{x_2=y_2=0\}} \psi(x)P(x,\cdot)\Omega^{\textup{top}}_x
\end{equation}
are K\"{a}hler–polarized on $\{x_2=y_2=0\}$ — this Hilbert space is generated by the sections
\begin{equation}\label{gene}
    \{w\mapsto P(z,w): z\in Y\}\;.
\end{equation}
In particular, for $w\in\{x_2=y_2=0\},$ the sections of \ref{gene} are polarized along 
\begin{equation}
\partial_{x_1}+i\partial_{y_1}\;.
\end{equation}
On the other hand, $M=\{x_2=y_1=0\}$ is Lagrangian for the real part of $\Omega$ and symplectic for its imaginary part.  \Cref{conv} doesn't hold on $M=\{x_2=y_1=0\}$ and therefore the integral doesn't converge — nevertheless, it is true that for $w\in \{x_2=y_1=0\}$ the sections of \ref{gene} are polarized along the real polarization
\begin{equation}
    \partial_{x_1}+\partial_{y_2}\;.
\end{equation}
This is the $+1$ eigenbundle for $K\vert_M,$ where 
\begin{align}
    &\nonumber I(\partial_{x_1})=\partial_{x_2}\,,\, I(\partial_{y_1})=\partial_{y_2}\;,
    \\&\nonumber J(\partial_{x_1})=\partial_{y_1}\,,\, J(\partial_{x_2})=\partial_{y_2}\;,
    \\& K(\partial_{x_1})=\partial_{y_2}\,,\,K(\partial_{x_2})=-\partial_{y_1}\;.
    \end{align}
\end{exmp}
$\,$\\The next example is motivated by the discussion surrounding equation 2.19 of \cite{brane}.
\begin{exmp}
We consider the example with $(Y,\Omega)$ given by
\begin{equation}
 S^2_{\mathbb{C}}:=\{(x,y,z)\in\mathbb{C}^3:x^2+y^2+z^2=1\}\;,\;\;\Omega=\frac{dx\wedge dy}{z}
\end{equation}  
which is isomorphic to $\textup{T}^*\mathbb{P}^1.$ We identify the tangent space at $(x,y,z)\in S^2_{\mathbb{C}}$ with all $(a,b,c)\in\mathbb{C}^3$ such that \begin{equation}
    (x,y,z)\cdot (a,b,c):=xa+by+cz=0\;.
\end{equation} 
The almost complex structure $I$ is given by multiplication by $i,$ while $J$ is given by the complexified cross product
\begin{equation}
    J_{(x,y,z)}(a,b,c)=(x,y,z)\times (a,b,c)=(yc-zb,za-xc,xb-ya)\;.
\end{equation}
Two choices for $M$ are $S^2,$ which corresponds to $x,y,z\in\mathbb{R},$ and the subvariety given by $x,y\in i\mathbb{R},\,z\in\mathbb{R}.$  The latter is the hyperboloid model of the unit disk $\mathbb{D}.$ These are both $J$–complex manifolds, but only $S^2$ is K\"{a}hler. The unit disk is such that $\Omega(JX,X)<0,$ ie. $-J$ is K\"{a}hler for $\mathbb{D}.$\footnote{On $\mathbb{D},$ $-J$ is the hyperk\"{a}hler almost complex structure.}
\\\\We can quantize $S^2$ and $\mathbb{D}$ using the Toeplitz quantization. We get quantization maps
\begin{align}
q_{S^2}:S^2\to\mathbb{P}\mathcal{H}_{S^2}\,,\;\;
q_{\mathbb{D}}:\mathbb{D}\to\mathbb{P}\mathcal{H}_{\mathbb{D}}\;,
\end{align}
and these analyically continue to
\begin{equation}
q_{S^2}:S_{\mathbb{C}}^2\to\textup{T}^*\mathbb{P}\mathcal{H}_{S^2}\,,\;\;
q_{\mathbb{D}}:S_{\mathbb{C}}^2\to\textup{T}^*\mathbb{P}\mathcal{H}_{\mathbb{D}}\;.
\end{equation}
Equivalently, we can analytically continue their Bergman kernels (ie.\ the orthogonal projections onto K\"{a}hler–polarized sections, see example 6.18 and section 9.3 of \cite{Lackman1}). Locally, on $S^2\backslash\infty\xhookrightarrow{} S^2_{\mathbb{C}},$ the idempotent corresponding to $M=S^2$ is given by
\begin{equation}
P_{S^2}(z_1,z_2)=\frac{1+\bar{z}_1z_2}{1+|z_1|^2}\;.
\end{equation} The two Hilbert spaces are different but we can still compare the $3$–point functions.\ A simple computation gives
\begin{equation}
\Delta_{S^2}\widebar{\Delta}_{\mathbb{D}}=1\;.
\end{equation}
This is related to the fact that the analytic continuations of the K\"{a}hler almost complex structures of $S^2$ and $\mathbb{D}$ differ by a sign. In a sense, quantizing $\mathbb{D}$ with a bad choice of almost complex structure is equivalent to quantizing $S^2.$
\end{exmp}

\begin{thebibliography}{9}
\bibitem{ber1}
F.A. Berezin. \textit{General concept of quantization.} Commun.Math. Phys. 40, 153–174 (1975). https://doi.org/10.1007/BF01609397
\bibitem{biq}
Olivier Biquard and Paul Gauduchon. \textit{Hyperk\"{a}hler Metrics on Cotangent Bundles of Hermitian Symmetric Spaces.} Geometry and Physics (1996). ISBN:9781003072393
\bibitem{bord}
M. Bordeman, E. Meinrenken and M. Schlichenmaier. \textit{Toeplitz quantization of K\"{a}hler 
manifolds and gl(n), n $\to\infty$ 
limits.} Comm. Math. Phys. 165 (1994), 281-296.
\bibitem{bor}
David Borthwick and Alejandro Uribe. \textit{Almost complex structures and geometric quantization.} Mathematical Research Letters 3 (1996): 845-861.
\bibitem{cahen}
Michel Cahen, Simone Gutt, John Rawnsley.
\textit{Quantization of K\"{a}hler Manifolds. II}
Transactions of the American Mathematical Society, Vol. 337, No. 1 (May, 1993), pp. 73-98 (26 pages)
https://doi.org/10.2307/2154310
\bibitem{catt}
 A. S. Cattaneo and G. Felder. \textit{A Path Integral Approach to the Kontsevich Quantization Formula}. Comm Math Phys 212, 591–611 (2000). https://doi.org/10.1007/s002200000229
\bibitem{klauder}
Ingrid Daubechies and John R. Klauder. \textit{Quantum-mechanical path integrals with Wiener measure for all polynomial Hamiltonians. II} J. Math. Phys. 26, 2239 (1985); doi: 10.1063/1.526803 
\bibitem{brane}
D. Gaiotto and E. Witten. \textit{Probing Quantization Via Branes.} \href{https://arxiv.org/abs/2107.12251v2}{arXiv:2107.12251} (2021).
\bibitem{kibble}
Tom W. B. Kibble. \textit{Geometrization of quantum mechanics.} Communications in Mathematical Physics 65 (1979), 189–201. doi:10.1007/BF01225149
\bibitem{klauder2}
John R. Klauder. \textit{Quantization is geometry, after all.} Annals of Physics Volume 188, Issue 1, 15 November 1988, Pages 120-141.
\bibitem{klauder4}
J.R. Klauder and B.-S. Skagerstam. \textit{A Coherent-State Primer.} World Scientific Publishing (1985), Singapore, 1-115.
\bibitem{kontsevich}
M. Kontsevich. \textit{Deformation Quantization of Poisson Manifolds.} Letters in Mathematical Physics 66, 157–216 (2003). https://doi.org/10.1023/B:MATH.0000027508.00421.bf
\bibitem{kor}
Y.A. Kordyukov. \textit{Berezin–Toeplitz Quantization on Symplectic Manifolds of Bounded Geometry.} Math Notes 112, 576–587 (2022). https://doi.org/10.1134/S0001434622090267
\bibitem{Lackman1}
Joshua Lackman. \textit{On an Axiomatization of Path Integral Quantization and its Equivalence to Berezin's Quantization.} \href{https://arxiv.org/abs/2410.02739}{arXiv:2410.02739} [math.SG] (2024).
\bibitem{poland0}
A. Odzijewicz. \textit{Coherent states and geometric quantization.} Commun.Math. Phys. 150, 385–413 (1992). https://doi.org/10.1007/BF02096666
\bibitem{mar}
Martin Schlichenmaier. \textit{Berezin-Toeplitz Quantization for Compact Kähler Manifolds. A Review of Results.} Advances in Mathematical Physics (2010). https://doi.org/10.1155/2010/927280.
\end{thebibliography}
\end{document}